\documentclass[11pt]{article}
\usepackage{amsmath}
\usepackage{amsthm}
\usepackage{amssymb}
\usepackage{latexsym}
\usepackage{lscape}
\usepackage{epsfig}
\usepackage{pstricks}
\usepackage{enumerate}
\usepackage{exscale}
\usepackage{relsize}
\usepackage{amsmath}
\usepackage{indentfirst}
 \usepackage{graphicx} 
\usepackage{epstopdf}
\usepackage{subfigure}
\usepackage{lineno}

\newtheorem{theorem}{Theorem}[section]
\newtheorem{lemma}[theorem]{Lemma}
\newtheorem{proposition}[theorem]{Proposition}

\theoremstyle{definition}
\newtheorem{definition}{Definition}

\theoremstyle{remark}
\newtheorem{remark}{Remark}

\begin{document}

\title{On shape sphere and rotation of three-body motion}
\author{
Wentian Kuang \thanks{Partially supported by NSFC(No.11901279). E-mail: kuangwt@sdu.edu.cn}\\
School of Mathematics, Shandong University\\
Jinan 250100, P.R. China
}
\date{}

\maketitle

\begin{abstract}
For three-body problem, R.Montgomery \cite{Mont1996} proved a reconstruction formula which calculates the overall rotation relating two similar triangle configurations if the initial triangular configuration is similar to the configuration formed at some later time. In this paper, we extend the formula so that it gives the angle of rotation for a single particle without requiring the similarity of initial and final configurations. The proof is different from that in \cite{Mont1996} and uses fundamental calculus. Moreover, we answered a question proposed in \cite{Mont1996}.
\end{abstract}

{\bf AMS classification number:} 70F07, 70E55\\

\section{Introduction}
Let $q_i\in \mathbb{R}^d(i=1,\dots,n)$ be $n$ particles. Each particle $q_i$ is endowed with a mass $m_i>0$. Configuration space is the set of all $n$-tuples $q=(q_1,\dots,q_n)$ with $q_i\in \mathbb{R}^d$. A motion of the $n$-particle system in time interval $[a,b]$ is a continuous path $q\in C([a,b],\mathbb{R}^{d\times n})$. If $q$ is smooth, then the velocity of system $v=\dot{q}$ is given by the time derivative of $q$. In the following of this paper, we always assume that any motion is smooth so that we can get rid of problems about regularity. Our results extends to piecewise smooth path readily.

It's well-known that translations can be separated from other motions. The velocity of a translation satisfies $v_1=v_2=\dots=v_n$ which represent the linear momentum of motion. For a conservative system, linear momentum is preserved by the motion. As in most papers, we assume that the center of mass is always at origin and the configuration space restricted to the following set
\[\Sigma=\Big\{q\in \mathbb{R}^{d\times n}\Big| \sum_{1}^{n}m_iq_i=0\Big\}.\]

Elements in $G=SO(d)$ act on $\Sigma$ by rotations in $\mathbb{R}^d$. Following \cite{Gu} of A.Guichardet, velocity at any configuration $q\in \Sigma$ can be decomposed into two parts: the rotation part $V_R$ and the internal part $V_I$. They are 
\begin{equation}\label{eqdecompose}
\begin{split}
V_R&=\Big\{v\in T_q(G\cdot q)\Big| G\cdot q\  \text{is the $G$-orbit of $q$}\Big\};\\
V_I&=\Big\{v\in (\mathbb{R}^d)^n\Big|\sum_{1}^{n} m_iq_i\wedge v_i=\sum_{1}^{n} m_iv_i=0\Big\}.
\end{split}
\end{equation}
Intuitively, the $V_R$ part follows from action of $SO(d)$ and $V_I$ part follows from scaling and changing of the similarity classes of configuration. If $d=2$ or $3$, clearly $v\in V_I$ if and only if the motion has angular momentum $0$ and $V_R$ is isomorphic to the space of admissible angular momentum at $q$.

Given a motion $q:[0,T]\rightarrow \Sigma$ such that the beginning configuration and the final configuration are similar, then they differ by a rotation(i.e. an elements in $SO(d)$). A basic question is: what can one tell about the rotation relating two boundary configurations? A.Guichardet in \cite{Gu} proved that if $n>d$, then every two points of an arbitrary $G$-orbit can be joint by a smooth \textsl{vibrational curve}(a motion such that $v\in V_I$ all the time). Thus both the rotation part and the internal part will contribute to the angle of rotation.

In this paper, we will focus on $n=3$ while $d=2$ or $3$. This question has been studied in \cite{Mont1996,Iw}. For planar $3$-body problem, all the oriented similar classes of configuration form a two sphere called by "shape sphere". Since the beginning configuration and final configuration are similar, they differs by a rotation of the plane and the motion projects to a closed curve on shape sphere. Then the angle of rotation is given by the following formula
\begin{equation}\label{eqplanar}
\Delta\theta=\int_{0}^{T}\frac{J}{I}+2\int_{D_0}\Omega_0,
\end{equation}
where $J$ is angular momentum and $I$ is moment of inertia. $D_0$ is the disk on shape sphere bounded by this closed curve and $\Omega_0$ is the standard area form on shape sphere.

For $d=3$, it not clear what is the angle of rotation between two similar oriented configuration. A natural idea is to put the two configurations on a plane by a certain way, as Montgomery did in \cite{Mont1996}. He assumed that the angular momentum is preserved by the motion. Denote $\mathbf{e}$ the axis of angular momentum and $\mathbf{n}(t)$ is the normal vector for $q(t)$. After rotating the normal vectors $\mathbf{n}(0)$ and $\mathbf{n}(T)$ to $\mathbf{e}$ along minimizing geodesics on sphere, the angle of rotation $\Delta\theta$ can be defined as planar case. Montgomery \cite{Mont1996} proved the following formula, which is an extension of \eqref{eqplanar}. 
\begin{equation}\label{eqmont}
\Delta\theta=\int_{0}^{T}\omega(t)+\int_{D}\Omega,
\end{equation}
where $\omega(t)$ depends on angular momentum, moment of inertia tensor and the axis of angular momentum, $\Omega$ is a two form on a "reduced configuration space" and $D$ is a disk bounded by a closed curve in reduced configuration space defined by the motion. Readers can refer to \cite{Mont1996} for details. 

In Montgomery's paper, he enlarged $\Sigma$ to form a set $\widetilde{\Sigma}$, the space of oriented configurations. An element in $\widetilde{\Sigma}$ is a configuration in $\Sigma$ together with a unit vector $\mathbf{n}$ which is orthogonal to the subspace spanned by vertices of $q$. This step is essential and the proof in \cite{Mont1996} rely on the geometry of $\widetilde{\Sigma}$. He noticed that it may be possible to carry out all calculations directly on $\Sigma/SO(2)$ and asked the following question: can reconstruction formula and calculations be reformulated solely in terms of $\Sigma/SO(2)$? Here $SO(2)$ is the rotation group about axis $\mathbf{e}$.

The problem for spatial case faces some difficulties which do not appear in planar case. Firstly, there is no natural orientation for a spatial configuration. Different ways of choosing normal vector will affects the result. Secondly, if $\mathbf{n}=\mathbf{-e}$, the minimizing geodesic from $\mathbf{n}$ to $\mathbf{e}$ is not unique. Finally, given velocity at a collinear configuration, we can not deduce how the configuration rotate. For these reasons, Montgomery had to worked on the space $\widetilde{\Sigma}$.

In this paper, we extend the formula \eqref{eqplanar} and \eqref{eqmont} to calculate the angle of rotation for any particle without requiring the conservation of angular momentum or the similarity of initial and final configuration. In fact, given the initial  configuration $q(0)$, angular momentum $J(t)$, momentum of inertia $I(t)$ and the projected curve on shape sphere $\gamma(t)$, the motion is uniquely defined. One should be able to known the rotation of each particle even though two boundary configurations are not similar. The proof in this paper is totally different from that in \cite{Mont1996}. All we need are some facts about shape sphere and fundamental calculus. The reconstruction formula considered in \cite{Mont1996} is a special case of our result. As a corollary, we give a positive answer to the question proposed by Montgomery in \cite{Mont1996}.

If the particle we concern does not go through origin, the angle of rotation is uniquely defined by polar coordinates. If it goes across the origin, then we should take a perturbation. The second term of formula may differ by $2\pi$ depending on the perturbation. We regard two angles of rotation the same if they differ by $2k\pi$ for $k\in \mathbb{Z}$. Thus the angle of rotation only depends on initial and final configurations.

We state our result for planar case and spatial case separately. For planar case, the angle of rotation is obviously defined by polar coordinates. Our result is the following theorem.

\begin{theorem}\label{thmplanar}
(i) For a piecewise smooth planar three-body motion $q(t)$ in time interval $[0,T]$. Let $\gamma(t)$ be the projected curve on shape sphere. Assume  $q_1(0)$ and $q_1(T)$ are away from origin. Then the angle of rotation from $q_1(0)$ to $q_1(T)$ is given by
	\begin{equation}\label{reconstruction}
	\Delta\theta=\int_0^T\frac{J}{I} \, dt+2\int_{D_\gamma}\Omega_0,
	\end{equation}
	where $J$ is the angular momentum, $I$ is the moment of inertia, $D_\gamma$ is a disk on shape sphere bounded by three curves: $\gamma(t)$, geodesic from $\gamma(0)$ to $C_1$, geodesic from $\gamma(T)$ to $C_1$. $\Omega_0$ is the standard area form on shape sphere so that the area of a disc is positive when the boundary curve rotates clockwise.\\
(ii) Let $q(t)$ and $\gamma(t)$ be the same as in (i), $Z_1(t)=q_3(t)-q_2(t)$. Assume $Z_1(0)$ and $Z_1(T)$ are away from origin. Then the angle of rotation from $Z_1(0)$ to $Z_1(T)$ is given by
\begin{equation}\label{reconstruction2}
\Delta\theta'=\int_0^T\frac{J}{I} \, dt+2\int_{D_\gamma'}\Omega_0,
\end{equation}
where $D_\gamma'$  is a disk on shape sphere bounded by three curves: $\gamma(t)$, geodesic from $\gamma(0)$ to $O_1$, geodesic from $\gamma(T)$ to $O_1$. 
\end{theorem}

For spatial case, let $\mathbf{e}$ be an unit vector and $X$ is the plane orthogonal to $\mathbf{e}$ in $\mathbb{R}^3$. Any oriented triangle configuration is defined by a triangle configuration $q$ and a normal vector $\mathbf{n}$ which gives the orientation. Whenever $\mathbf{n}\ne \mathbf{-e}$, there is a unique minimizing geodesic on unit sphere connecting $\mathbf{n}$ and $\mathbf{e}$. By moving $\mathbf{n}$ to $\mathbf{e}$ along the minimizing geodesic on the unit sphere, we get a configuration in $X$. Thus we can define a map $P: (q,\mathbf{n})\rightarrow x$ which maps a spatial configuration to a configuration in plane $X$.  Then the angle of rotation is defined to be that of $x(t)$ as in planar case. This is the same as what R.Montgomery did in \cite{Mont1996}. In that paper, the angular momentum is assumed to be preserved and $\mathbf{e}$ is chosen to be the axis of angular momentum. 

To get a reconstruction formula in spatial case, we need some assumption for motion. Denote
\begin{equation}\label{eqbad}
\Delta=\Big\{t\in [0,T]\Big|q(t)\  \text{is collinear} , J(t)\ne 0, \mathbf{e}\  \text{is not orthognal to the axis of configuration} \Big\}
\end{equation}
We call a collinear configuration $q(t)$ as above is "bad". Our theorem for spatial case can be stated as follows.

\begin{theorem}\label{thmspatial}
	Let $q(t)$ be a smooth $3$-body motion in $\mathbb{R}^3$ in time interval $[0,T]$. Assume $q(t)$ satisfies following conditions: (i) $\mathbf{n}(t)$ is smooth; (ii) for any $t_0$ such that $\mathbf{n}(t_0)=-\mathbf{e}$, it holds $\dot{\mathbf{n}}(t_0)\ne 0$; (iii) $\Delta$ has a zero measure. Then $x(t)=P(q(t),\mathbf{n}(t))$ is a smooth motion in the plane $X$. 
	
	Let $\gamma(t)$ be the projected curve on shape sphere, $\eta_1(t)$ be the angle coordinate of $P(q_1(t),\mathbf{n}(t))$ in $X$. Then 
	\begin{equation}\label{reconstructionspatial}
	\Delta\eta_1=\eta_1(T)-\eta_1(0)=\int_0^TF(J(t))dt+2\int_{D_\gamma}\Omega_0,
	\end{equation}
	where the definition of $F(J(t))$ can be found in the proof in section $4$, the second term on the right of \eqref{reconstructionspatial} is the same as in \eqref{reconstruction}.
\end{theorem}

\begin{remark}\label{rm1.1}
	(i)When $\mathbf{n}(t)=\mathbf{e}$ all the time, it is the planar case and the first term on the right of \eqref{reconstructionspatial} is the same as in \eqref{reconstruction}. Because $\mathbf{n}$ is not well-defined at collinear configuration and $P(q,\mathbf{n})$ is not well-defined when $\mathbf{n}=\mathbf{-e}$, the assumptions are needed to ensure that $x(t)$ is smooth. These assumptions are automatically satisfied for a generic motion.
	
	(ii)Note that the shape sphere is space of similarity classes of oriented triangles. If $q(t)$ is a triangular configuration, then one of $\gamma(t)$ and $\mathbf{n}(t)$ defines the other.
	
	(iii)There is also a result similar to \eqref{reconstruction2}, we omit the statement here.
\end{remark}

Since velocity can be decomposed into two parts, the strategy is to study how the two parts effect the angle velocity of a particle. Because we use a different approach from \cite{Mont1996}, formula \eqref{reconstructionspatial} is also difference from \eqref{eqmont}. In section \ref{sec4}, a map $\sigma$ is defined from $Lie(G)$ to angular momentum $J$. $\omega(t)$ in \eqref{eqmont} is actually $\mathbf{e}\cdot \sigma^{-1}(J(t))$, while $\sigma_{\mathbf{e}}^{-1}(J(t))$ and $\sigma_{\mathbf{n}}^{-1}(J(t))$ are given by direct decomposition of $ \sigma^{-1}(J(t))$ about a basis define by $\mathbf{e}$ and $\mathbf{n}$. The second term in the right of\eqref{eqmont} is an integral on reduced configuration space, which is a two-sphere bundle over the shape sphere. The second term in our formula is an integral on the shape sphere.

The paper is organized as follows. In section $2$, we introduce some necessary foundations about shape sphere. In section $3$, we prove reconstruction formula for planar case. As an application, we give a characterization for the shape sphere. In section $4$, we prove the reconstruction formula for spacial case. In last section, we apply the formula to Montgomery's case and give a positive answer to the question proposed in \cite{Mont1996}.

\section{Introduction to shape sphere}
In this section, we give an introduction to the shape sphere for $3$ body problem. We only list some necessary results for later proof of main theorem. For more details readers can refer to \cite{Mont2015}. 

For planar case, our configuration space is
\begin{equation}\label{configuration}
\Sigma=\Big\{ q \in \mathbb{C}^{3} \, \bigg{|} \, \sum_{i=1}^3 m_iq_i =0 \Big \}.
\end{equation}
\begin{definition}
Two configurations in $\Sigma$ are oriented congruent if there is a rotation taking one  to the other. \textsl{Shape space} is the space of oriented congruence classes.
\end{definition}

\begin{definition}
	For the planar three-body problem, the \textsl{Jacobi coordinates} $(\tilde{Z}_1,\, \tilde{Z}_2) \in \mathbb{C}^{2}$ and the \textsl{normalized Jacobi coordinates} $(Z_1,\, Z_2) \in \mathbb{C}^{2}$ with respect to $q_1$ are given by
	\begin{equation}\label{jacobi}
	\begin{aligned}
	\tilde{Z}_1&=q_3-q_2,\ \ \ \ \tilde{Z}_2=q_1-\frac{m_2q_2+m_3q_3}{m_2+m_3},\\
	Z_1&=\mu_1(q_3-q_2),\ \ \ \ Z_2=\mu_2\Big(q_1-\frac{m_2q_2+m_3q_3}{m_2+m_3}\Big),
	\end{aligned}
	\end{equation}
	where $\frac{1}{\mu_1^2}=\frac{1}{m_2}+\frac{1}{m_3}$ and $\frac{1}{\mu_2^2}=\frac{1}{m_1}+\frac{1}{m_2+m_3}$.
\end{definition}
By normalizing, one can greatly simplified the notations. Clearly, the normalized Jacobi coordinates $(Z_1,Z_2)$ is one-to-one corresponded to the centered configuration $q\in \Sigma$. The following fact can be proved by direct computation.

\begin{lemma}\label{lemmaIJ}
	The momentum of inertia $I=|Z_1|^2+|Z_2|^2$ and the angular momentum is $J=-Im(Z_1 \overline{\dot{Z}}_1+Z_2\overline{\dot{Z}}_2).$
\end{lemma}

Let $\omega_i\in \mathbb{R} \, (i=1,2,3,4)$ be four real variables satisfying
\begin{equation}\label{omega}
\left\{\begin{aligned}
\omega_4+\omega_1&=|Z_1|^2 \\
\omega_4-\omega_1&=|Z_2|^2 \\
\omega_2+\sqrt{-1}\omega_3&=\overline{Z}_1Z_2
\end{aligned}
\right.
\end{equation}

\begin{remark}
	The definition of normalized Jacobi coordinates and $\omega_i$s in \eqref{jacobi} and \eqref{omega} are carefully chosen.  In this way, the Lagrangian configuration with the particles $1,2,3$ rotating counterclockwise is on the upper hemisphere.  It's also convenient for later use.
\end{remark}

Define the map $\pi$ from $\Sigma$ to $\mathbb{R}^3$ by
\begin{equation*}
\pi :(q_1,q_2,q_3)\rightarrow (Z_1, Z_2) \rightarrow (\omega_1,\omega_2,\omega_3).
\end{equation*}
The following theorem is from  \cite{Mont2015}.
\begin{theorem}\label{thmshape}
	Shape space is homeomorphic to $\mathbb{R}^3$. The quotient map $\pi$ enjoys the following properties:\\
	(a) Two triangles $q, q' \in \mathbb{C}^3$ are oriented congruent if and only if $\pi(q) = \pi(q')$.\\
	(b) $\pi$ is onto.\\
	(c) $\pi$ projects the triple collision locus onto the origin.\\
	(d) $\pi$ projects the locus of collinear triangles onto the plane $\omega_3 = 0$, where
	$(\omega_1,\omega_2,\omega_3)$ are standard linear coordinates on $\mathbb{R}^3$. Moreover, $\omega_3$ is the
	signed area of the corresponding triangle, up to a mass-dependent constant.\\
	(e) Let $\sigma : \mathbb{R}^3 \rightarrow \mathbb{R}^3$ be the reflection across the collinear plane: $\sigma(\omega_1,\omega_2,\omega_3)
	= (\omega_1,\omega_2,-\omega_3)$. Then the two triangles $q, q'  \in \mathbb{C}^3$ are congruent if and only
	if either $\pi(q) = \pi(q')$ or $\pi(q) = \sigma(\pi(q'))$.\\
	(f) $\omega_1^2+ \omega_2^2+\omega_3^2=\omega_4^2=(\frac{1}{2}I)^2$.\\
   (g) $R =\sqrt{I}$ is the shape space distance to triple collision.
\end{theorem}

By Theorem \ref{thmshape}, the shape space is isomorphic to $\mathbb{R}^3$ and the Euclidian norm of the vector $(\omega_1,\omega_2,\omega_3)$ is one half of the moment of inertia. By fixing the moment of inertia, we get a sphere in shape space. 

As commented in \cite{Mont2015}, ``the shape space is not isometric to $\mathbb{R}^3$: Shape space geometry is not Euclidean. However, the geometry does have spherical symmetry. Each sphere centered at triple collision is isometric to the standard sphere, up to a scale factor. We identify these spheres with the shape sphere."

\begin{definition}
	The \textsl{shape sphere} is the sphere in the shape space given by $I=1$, i.e. the sphere given by $\omega_1^2+\omega_2^2+\omega_3^2=\frac{1}{4}$.
\end{definition}

For our later use, we mark some points on the shape sphere.

\noindent \textbf{Notations:}  Let $C_1$ denote the collision configuration when $q_2=q_3$, $O_1$ denotes the configuration when $q_1=0$. $C_2$, $O_2$, $C_3$ and $O_3$ denote the corresponding configurations respectively. Let $P_1$ and $P_2$ denote the corresponding pole on the upper and lower hemisphere. Let $L_1$ denote the Lagrangian configuration when body $1,2,3$ rotate counterclockwise, $L_2$ is the other Lagrangian configuration with opposite orientation. Let $E_i(i=1,2,3)$ denote the corresponding Euler configuration. In general, $E_i\ne O_i(i=1,2,3)$ and $L_j\ne P_j(j=1,2)$. At the end of section $3$, figure \ref{2} shows how these marked points are located on the shape sphere.
%
%
%

By \eqref{jacobi} and \eqref{omega}, $O_1$ is $(1/2,0,0)$ and $C_1$ is $(-1/2,0,0)$. Clearly, we have different ways to define Jacobi coordinates. Different choice of Jacobi coordinates differ by reflections or rotations of shape space. For example, by letting $\tilde{Z}_1=q_2-q_3$ in \eqref{jacobi} the resulting shape sphere differs to original one by a reflection with respect to $\omega_2,\omega_3$-plane. If we define Jacobi coordinates by $\tilde{Z}_1=q_2-q_1$ and $\tilde{Z}_2=q_3-\frac{m_2q_2+m_1q_1}{m_2+m_1}$, then the resulting shape sphere differs to original one by a rotation such that $O_3$ is $(1/2,0,0)$ and $C_3$ is $(-1/2,0,0)$. In the following of this paper, we will always use the original coordinate system defined by \eqref{jacobi} and \eqref{omega} unless explicitly told. 

We can write $(\omega_2,\omega_3)$ in polar coordinates $(r,\xi)$. Then $\xi$ is well defined on $S\backslash \{O_1,C_1\}$ and $\omega_2+\sqrt{-1}\omega_3=re^{i\xi}$. Let $Z_1=r_1e^{i\xi_1}$ and $Z_2=r_2e^{i\xi_2}$. By \eqref{omega}, it follows that
\[\omega_2+\sqrt{-1}\omega_3=r_1r_2e^{i(\xi_2-\xi_1)}.\]
Hence, there exists some integer $k$, such that
\[r=r_1r_2,\ \ \ \ \ \xi=\xi_2-\xi_1+2k\pi.\]

Without loss of generality, we can assume $k=0$. Then $\xi=\xi_2-\xi_1$ is the angle of rotation from $Z_1$ to $Z_2$. $C_1$ is a singular point for $\xi_1$ while $O_1$ is a singular point for $\xi_2$. Thus we have

\begin{proposition}\label{proangle}
Given any $\xi\in [0,2\pi)$, let $l_\xi$ be the meridian collecting $C_1$ and $O_1$. The projection of a configuration on the shape space locates on $l_\xi$ if and only if the angle of rotation from $Z_1$ to $Z_2$ is $\xi$.
\end{proposition}

\section{Reconstruction formula for planar case}
In this section, we give a prove of formula \eqref{reconstruction}. Here we use a different approach by fundamental calculus which relies little on geometry.

In this section, the configuration space is 
\[\Sigma=\Big\{q\in \mathbb{C}^3\Big| \sum_{1}^{n}m_iq_i=0\Big\}.\]

\begin{lemma}\label{lemma0angular}
	For any parameterized curve $\gamma: [0,s]\rightarrow \mathbb{R}^3$ in the shape space away from triple collision and a configuration $q(0)$ realizing $\gamma(0)$, there is a unique zero-angular-momentum motion $q:[0,\, s]\rightarrow \Sigma$ realizing $\gamma$. 
\end{lemma}
\begin{proof}
	Let $Z_1=r_1e^{i\xi_1},  Z_2=r_2e^{i\xi_2}$ be the normalized Jacobi coordinates of $q$. Now $Z_1(0)$ and $Z_2(0)$ are given, we want to solve $(Z_1(t),Z_2(t))$ from $\gamma(t)$.
	
	Since the $\gamma(t)\in \mathbb{R}^3$ is given, we can solve $r_1(t),r_2(t)$ and $\xi(t)=\xi_2(t)-\xi_1(t)$ by \eqref{omega}. It should be reminded that $\xi_1$ and $\xi_2$ is not defined at origin. Since $\gamma(t)$ is away from triple collision, either $\xi_1(t)$ or $\xi_2(t)$ is well defined thus defines the configuration.  The rest to do is to find $\xi_1(t)$ or $\xi_2(t)$. By Lemma \ref{lemmaIJ}, the angular momentum is 
	\begin{equation}
	\begin{split}
	& r_1^2(t)\dot{\xi}_1(t)+r_2^2(t)\dot{\xi}_2(t)\\
	&= r_1^2(t)\dot{\xi}_1(t)+r_2^2(t)\dot{\xi}_1(t)+r_2^2(t)\dot{\xi}(t)\\
	&= I(t)\dot{\xi}_1(t)+r_2^2(t)\dot{\xi}(t)\\
	&= 0.
	\end{split}		
	\end{equation}
	Whenever $\xi_1(t)$ is well defined, we have \[\dot{\xi}_1(t)=-\frac{r_2^2(t)}{I(t)}=-\frac{r_2^2(t)\dot{\xi}(t)}{r_1^2(t)+r_2^2(t)}.\]
	Now $\xi_1(0)$ is given, $\xi_1(t)=\xi_1(0)+\int_{0}^{t}\dot{\xi}_1(s)ds$ until $Z_1$ goes to origin where $\xi_1$ is not defined. For this case, we can use $\xi_2(t)$ instead. The proof is complete!
\end{proof}

\begin{lemma}\label{lemmameridian}
	Let $q:[0,s]\rightarrow \Sigma\backslash\{C_1,O_1\}$ be a motion with zero angular momentum and $Z_1=r_1e^{i\xi_1},  Z_2=r_2e^{i\xi_2}$ be the normalized Jacobi coordinates of $q$. If the projected path on the shape sphere is part of a meridian $l_\xi$ connecting $C_1$ and $O_1$ for some $\xi\in [0,2\pi)$, then $\dot{\xi}_1=\dot{\xi}_2=0$.
\end{lemma}
\begin{proof}
	By the assumption, $\xi_1, \xi_2$ is well defined and
    \[\xi_2-\xi_1=\xi.\]
	It implies that 
	\[\dot{\xi}_2-\dot{\xi}_1=0.\]
	By the zero angular momentum assumption, we have
	\[J = Z_1\times \dot{Z}_1+Z_2\times \dot{Z}_2=r_1^2\dot{\xi}_1+r_2^2\dot{\xi}_2=0.\]
	Solve $\dot{\xi}_1$ and $\dot{\xi}_2$ in the two above equations, we have
	\[\dot{\xi}_1=\dot{\xi}_2=0.\]
\end{proof}
\begin{remark}
	Actually, \eqref{reconstruction} is a direct consequence of formula \eqref{eqplanar} and Lemma \ref{lemmameridian}. Because the geodesics on shape sphere from $\gamma(0)$ to $C_1$ and from $\gamma(T)$ to $C_1$ keeps $\xi_1(t)$ fixed.
\end{remark}

\textbf{Proof of Theorem \ref{thmplanar}:} For simplicity, we prove the theorem when $\xi_1,\xi_2$ and $\xi$ are well defined for all $t\in [0,T]$. That is $\gamma(t)\ne C_1$ and $\gamma(t)\ne O_1$. At the end, we'll see that this assumption can be removed. 

Given any motion $q:[0,\, s] \rightarrow \Sigma$, then the corresponding curve in shape space $\gamma(t)$ is known. So there is a unique zero-angular-momentum motion $q'(t)$ realizing $\gamma(t)$ and $q'(0)=q(0)$. Let $Z_1'=r_1'e^{i\xi_1'},Z_2'=r_2'e^{i\xi_2'}$ be Jacobi coordinates of $q'$. Clearly, we have
\begin{equation}
	\begin{split}
	r_1'(t)&=r_1(t),\ \ r_2'(t)=r_2(t),\\
	\xi_2'(t)&-\xi_1'(t)=\xi_2(t)-\xi_1(t).
	\end{split}
\end{equation}
Thus \[\dot{\xi}_2(t)-\dot{\xi}_2'(t)=\dot{\xi}_1(t)-\dot{\xi}_1'(t).\]
The angular momentum of $q(t)$ is 
\begin{equation}
\begin{split}
J(t)&=r_1^2(t)\dot{\xi}_1(t)+r_2^2(t)\dot{\xi}_2(t),\\
&= I(t)(\dot{\xi}_2(t)-\dot{\xi}_2'(t))+r_1^2(t)\dot{\xi}_1'(t)+r_2^2(t)\dot{\xi}_2'(t),\\
&= I(t)(\dot{\xi}_2(t)-\dot{\xi}_2'(t)).
\end{split}
\end{equation}
We have 
\begin{equation}
\dot{\xi}_2(t)-\dot{\xi}_2'(t)=\frac{J(t)}{I(t)}.
\end{equation}

Note that $\xi_2$ is just the angle of $q_1$ in polar coordinate. Hence 
\begin{equation}\label{eqrotationpart}
\begin{split}
\Delta\theta&=\xi_2(T)-\xi_2(0),\\
&=\int_{0}^{T}\dot{\xi}_2(t)dt,\\
&=\int_0^T\frac{J(t)}{I(t)}dt+\int_{0}^{T}\dot{\xi}_2'(t)dt,\\
&=\int_0^T\frac{J(t)}{I(t)}dt+\xi_2'(T)-\xi_2'(0).
\end{split}
\end{equation}

It is left to find $\xi_2'(T)-\xi_2'(0)$, which is the angle of rotation for the zero-angular-momentum motion $q':[0,T]\rightarrow \Sigma$. Now it suffice to consider motions with zero angular momentum.

For simplicity of notation, we denote $q$ to be a zero-angular-momentum motion(i.e. $q'$ above is now replaced by $q$).  Since we are interest in the angle of rotation  $\xi_2(T)-\xi_2(0)$, by normalizing the momentum of inertia,  it's enough to only consider the projected curves on the shape sphere. Thus we can also assume $\gamma(t)$ is on the shape sphere. 

Since
\[J(t)=r_1^2(t)\dot{\xi}_1(t)+r_2^2(t)\dot{\xi}_2(t)=0.\]
We have
\begin{equation}
\frac{\dot{\xi}_2(t)}{\dot{\xi}(t)}=\frac{\dot{\xi}_2(t)}{\dot{\xi}_2(t)-\dot{\xi}_1(t)}=\frac{r_1^2(t)}{r_1^2(t)+r_2^2(t)}=r_1^2(t).
\end{equation}

Let $\Omega_0$ be the standard area form on the shape sphere defined by $\omega_1^2+\omega_2^2+\omega_3^2=\frac{1}{4}$. Let $S(t)$ be the area enclosed by two meridians $l_{\xi(0)}$ and $l_{\xi(t)}$(Here we should choose boundary orientation so that $S(t)$ increases as $\xi(t)$ increasing). Let $D_\gamma(t)$ be the disk enclosed by a closed curve on shape sphere composed by three pieces: $\gamma(0)$ to $\gamma(t)$, $C_1$ to $\gamma(0)$ along $l_{\xi(0)}$ and $C_1$ to $\gamma(1)$ along $l_{\xi(t)}$. The orientation is chosen so that it coincide with $S(t)$ on $l_{\xi(0)}$ and $l_{\xi(t)}$. 

Note that $\xi$ is also the angle coordinate of $(\omega_2,\omega_3)$. There has 
\begin{equation}
\xi(t)-\xi(0)=2S(t),\ \ \ \dot{\xi}(t)=2\dot{S}(t).
\end{equation}
Let $R(t)=\int_{D_\gamma(t)}\Omega_0$ be the area of $D_\gamma(t)$. By direct computation from calculus, we have
\begin{equation}
\frac{\dot{R}(t)}{\dot{S}(t)}=\frac{1}{2}+\omega_1(t)=\frac{1}{2}(r_1^2(t)+r_2^2(t))+\frac{1}{2}(r_1^2(t)-r_2^2(t))=r_1^2(t).
\end{equation}
Actually, $\dot{R}(t)$ is given by rotation of the curve connecting $C_1$ to $\gamma(t)$ along $l_{\xi(t)}$ while $\dot{S}(t)$ is given by rotation of the whole meridian $l_{\xi(t)}$. Thus $\frac{\dot{R}(t)}{\dot{S}(t)}$ is ratio of the area where $\omega_1\le \omega_1(t)$ to the area of whole sphere, which is given by $\frac{1}{2}+\omega_1(t)$.

Summarizing three above equations, we have
\begin{equation}\label{eqdotxi}
\begin{split}
\frac{\dot{\xi}_2(t)}{\dot{R}(t)}&=\frac{\dot{\xi}(t)}{\dot{S}(t)}=2,\\
\dot\xi_2(t)&=2\dot{R}(t).
\end{split}	
\end{equation}
Integrating \eqref{eqdotxi} in $[0,T]$, we have
\begin{equation}\label{eqformularfor0angular}
	\xi_2(T)-\xi_2(0)=2(R(T)-R(0))=2\int_{D_\gamma}\Omega_0.
\end{equation}

Recall that  $\xi_2(T)-\xi_2(0)$ in \eqref{eqformularfor0angular} is actually $\xi_2'(T)-\xi_2'(0)$ in  \eqref{eqrotationpart}. Thus we have proved the reconstruction formula \eqref{reconstruction} for planar case under assumption that $\gamma(t)$ is away from $C_1$ and $O_1$.

If $\gamma(t)$ goes across $C_1$ or $O_1$, we can find a sequence of motion $q_{(n)}(t)$ away from $C_1,O_1$ converging to $q(t)$ in $C^1$ topology. Then we have uniform convergence for angular momentum, momentum of inertia and the projected curve on shape space. The first term of formula \eqref{reconstruction} converges uniformly while the second term may differ by $2\pi$ depending on how we choose the disk $D_{\gamma}$. This cause no confusion because a rotation of $2\pi$ is identity. Hence the assumption $\gamma(t)$ is away from $C_1$ and $O_1$ can be removed.

For piecewise smooth curves, the angle of rotation is clearly the angle sum of these smooth pieces. Thus we have proved $(i)$ of Theorem \ref{thmplanar}. $(ii)$ of Theorem \ref{thmplanar} can be proved in a similar way by applying above argument to $\xi_1$. Then $D_\gamma'$  is the disk on shape sphere bounded by three curves: $\gamma(t)$, geodesic from $\gamma(0)$ to $O_1$, geodesic from $\gamma(T)$ to $O_1$. 

\qed

In the end of this section, we give some characterization of the shape sphere as an application of Theorem \ref{thmplanar}.

Note that we've marked some points on the shape sphere. Among them, $L_1$,\, $L_2$,\, $E_1$,\, $E_2$,\, $E_3$ are the five central configurations, which are very important for the study of the planar three-body problem. Once the coordinate system is chosen, then the positions of these marked points is given which depends on masses of particles $(m_1,m_2,m_3)$. By our definition of the shape sphere, the six points $C_1$,$O_2$,$C_3$,$O_1$,$C_2$,$O_3$ arranged on the equator of the shape sphere counterclockwise and $C_i=-O_i(i=1,2,3)$. Let $(i,j,k)$ be a permutation of $(1,2,3)$, then $E_i$ is between $C_j$ and $C_k$. One can check that $E_i=O_i$ if and only if $m_j=m_k$, $E_i$ is between $C_j$ and $O_i$ if and only if $m_j>m_k$. When $m_1=m_2=m_3$, it's clear that $E_i=O_i (i=1,2,3)$ and $L_j=P_j (j=1,2)$ are the two poles on the shape sphere. The six points $E_i$s and $C_i$s is evenly distributed on the equator of the shape sphere. 

If the three masses are different, we can apply our result to find the relation of these points on the shape sphere. Assume $O_1=(1/2,0,0)$ and $C_1=(-1/2,0,0)$. Note that $P_1$ makes the area of triangle configuration largest under the condition $I=|Z_1|^2+|Z_2|^2=1$, $Z_1$ must be orthogonal to $Z_2$. For other versions of Jacobi coordinates, the corresponding two vector must also be orthogonal. Thus $P_1$ is characterized by the orientation (i.e. the positions of $1,2,3$ rotate counterclockwise) and the condition that the ortho-center coincides with the center of mass (i.e. the origin).

Consider $q:[0,T]\rightarrow \Sigma$ to be a special motion as in Fig \ref{1}. At $t=0$, $q(0)$ is a configuration corresponding to $P_1$. During the motion, $q_3$ is fixed and so is the center of mass of $q_1$ and $q_2$. $q_1$ and $q_2$ move uniformly linear towards their center of mass and collide at $t=T$. By Lemma \ref{lemmameridian}, it's a zero-angular-momentum motion. Its projected curve on the shape sphere is half of a meridian which goes from $P_1$ to $C_3$.

\begin{figure}
	\centering
	\includegraphics[height=7cm]{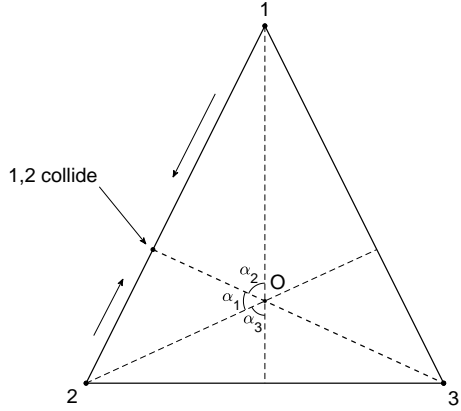}
	\caption{In the beginning, the triangle configuration corresponds to $P_1$. Body $3$ is fixed during the motion. Body $1$ and body $2$ move uniformly linear until they collide at the foot of the perpendicular. }
	\label{1}
\end{figure}

We can then compute the rotation of $q_1$ by Theorem \ref{thmplanar}. It is
\[2\int\int_{\Delta P_1C_3C_1}\Omega_0.\]
Here $\Delta P_1C_3C_1$ is the spherical triangle by closing the curve as in the proof of Theorem \ref{thmplanar}. One finds that the above integral is exactly half of the angle of rotation from $OC_1$ to $OC_3$. It's also the angle $\alpha_2$ in Fig.\ref{1}. A direct computation shows that
\[\cos \alpha_2=\sqrt{\frac{m_1m_3}{(m_1+m_2)(m_3+m_2)}}.\]
Hence we find the location of $C_3$ and $O_3=-C_3$.

Similarly, the angle of rotation from $OC_3$ to $OC_2$ is $2\alpha_1$,  the angle of rotation from $OC_2$ to $OC_1$ is $2\alpha_3$. They are given by
\[\cos \alpha_1=\sqrt{\frac{m_2m_3}{(m_2+m_1)(m_3+m_1)}}, \\
\cos \alpha_3=\sqrt{\frac{m_2m_1}{(m_3+m_2)(m_3+m_1)}}.\]

Since all $3$ body have positive masses, we have $\alpha_i\in (0,\pi/2)$, $\alpha_1+\alpha_2+\alpha_3=\pi$, $\alpha_i<\alpha_j$ if and only if $m_i<m_j$.  By our definition of the shape sphere, the six points $C_1$,$O_2$,$C_3$,$O_1$,$C_2$,$O_3$ arranged on the equator of the shape sphere counterclockwise. 

Let $(i,j,k)$ be a permutation of $(1,2,3)$, then $E_i$ is between $C_j$ and $C_k$. One can check that $E_i=O_i$ if and only if $m_j=m_k$, $E_i$ is between $C_j$ and $O_i$ if and only if $m_j>m_k$.

Next, we find the locations of $L_1$ and $L_2$ on the shape sphere. At the Lagrange point, the configuration is an equilateral triangle. Denote $\beta$ to be the angle of rotation from $Z_1$ to $Z_2$ of the equilateral triangle corresponding to $L_1$, then $\beta$ is a function of three masses. By Proposition \ref{proangle}, $L_1\in l_{\beta}$. Changing to another version of Jacobi coordinates, then there is a meridian $l_{\beta'}'$ from $C_3$ to $E_3$ such that $L_1\in l_{\beta'}'$. Thus the intersection point of $l_\beta$ and $l_{\beta'}'$ is $L_1$. The other Lagrangian point is given by $L_2=\sigma(L_1)$.

Fig.\ref{2} is an illustration for the case when $m_1<m_2<m_3$. 
\begin{figure}
	\centering
	\includegraphics[height=9.4cm,width=15.6cm]{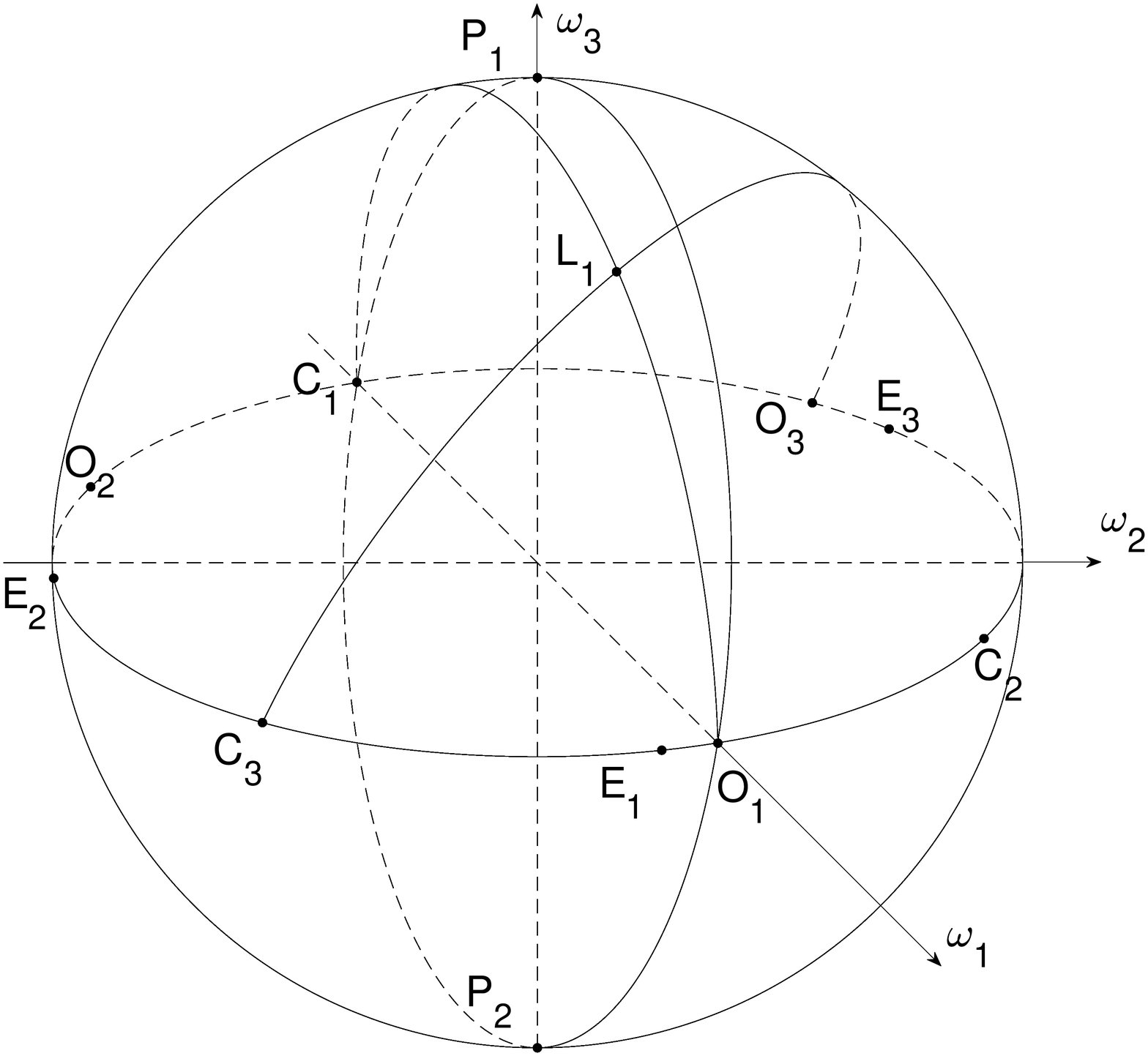}
	\caption{Marked points on the shape sphere for $m_1<m_2<m_3$}
	\label{2}
\end{figure}

\section{Reconstruction formula for spatial case}\label{sec4}
For planar case, it's easy to see what is the angle of rotation when two boundary configurations are similar. But it's not obvious for spatial case because two boundary configuration may not be in the same plane. Here we follow the idea of Montgomery \cite{Mont1996} .

Let $\mathbf{e}$ be an unit vector in $\mathbb{R}^3$, then $\mathbf{e}$ define a plane $X$.
\begin{equation*}
	X=\{x\in \mathbb{R}^3|x\cdot \mathbf{e}=0\}.
\end{equation*}
For planar configuration, a configuration uniquely defines a point on shape sphere. In $\mathbb{R}^3$, there is no nature orientation for a triangle configuration. Once a normal vector $\mathbf{n}$ is given, then $q$ uniquely defines a oriented triangle configuration hence a point on shape sphere. We move the spatial configuration $q$ to configuration on the plane $X$ in the following way: move $\mathbf{n}$ to $\mathbf{e}$ along the minimizing geodesic on unit sphere whenever $\mathbf{n}\ne -\mathbf{e}$. Thus we define a map $P$
\begin{equation}
	P: (q,\mathbf{n})\rightarrow x.
\end{equation}
where $q$ is a triangle configuration in $\mathbb{R}^3$, $x$ is the corresponding configuration on plane $X$. Clearly, $q$ and $x$ correspond to the same point on shape sphere.

The map $P$ can be restricted to one particle and $P(q_1,\mathbf{n})$ is a vector in $x_1\in X$. Let $\eta_1$ be the angle coordinate of $x_1$ in $X$, then the angle of rotation we seek for is $\eta_1(T)-\eta_1(0)$. Unfortunately, $P$ has some bad properties. $P$ is not well defined when $\mathbf{n}=-\mathbf{e}$ because the minimizing geodesic on unit sphere connecting two antipodal points  is not unique. One should also be careful about collinear configuration since the normal vector $\mathbf{n}$ can be chosen on a circle. For these reasons, we need assumptions in Theorem \ref{thmspatial} so that $P(q_1,\mathbf{n})$ is a continuous path in $X$.

Our strategy for the proof of reconstruction formula for spatial case is still to study the internal part and rotation part of motion separately. A.Guichardet in \cite{Gu} proved the following proposition.
\begin{proposition}[Proposition 3.1 in \cite{Gu}]\label{progui}
	If n=d and $q(t)$ is a smooth motion with zero angular momentum, then the hyperplane defined by $q(t)$ is constant.
\end{proposition}
This means that the $V_I$ part of velocity is tangent to the hyperplane defined by $q(t)$. We'll find that $V_I$ part can be treat similarly as in planar case. Thus the second term of \eqref{reconstruction} and \eqref{reconstructionspatial} are the same.

Now $G=SO(3)$, $Lie(G)$ can be identified with $\mathbb{R}^3$ so that for any vector $a\in \mathbb{R}^3$, $\frac{a}{|a|}$ is the axis of rotation and $|a|$ is angular velocity. For a given configuration, there is a obvious map 
\begin{equation}\label{eqsigma}
\sigma: \mathbb{R}^3\cong Lie(G)\rightarrow J\cong V_R
\end{equation}
where $J$ is the space of admissible angular momentum at $q$ and it's easy to see that $J\cong V_R$. Here $\sigma$ maps a vector to the corresponding angular momentum.

As pointed out in \cite{Gu,Mont1996}, $\sigma$ is a positive semi-definite symmetric linear map. $\sigma$ is an isomorphism when $q$ is a triangular configuration, while $\sigma$ has an one dimension kernel at a collinear configuration because any rotation about configuration axis keeps $q$ fixed. This is bad because different elements in $Lie(G)$ may give the same velocity and angular momentum. We should defined $\sigma^{-1}(J(t))$ to be $\frac{J(t)}{I(t)}$ at collinear configurations. Although $\sigma^{-1}$ is well defined away from collinear configurations, its norm $||\sigma^{-1}||$ goes to infinity as $q$ goes to collinear configuration. It means that a tiny angular momentum may cause a large angle velocity. And the integral $\eqref{reconstructionspatial}$ becomes singular at collinear configurations.

\textbf{Proof of Theorem \ref{thmspatial}:}  The assumptions $(i)$ and $(ii)$ of Theorem \ref{thmspatial} is to ensure that $P(q_1(t),\mathbf{n}(t))$ uniquely defines a smooth path $x(t)$ in $X$ if $\mathbf{n}(0)$ is given. Our goal is study the motion of $x(t)$. Still, we assume $q_1(t)$ is not the origin for $\forall t\in [0,T]$, so $\eta_1(t)$ is well-defined all the time. We need to study how $\eta_1(t)$ changes depending on the spatial motion.

If the projected curve on shape space $\gamma(t)$ is given, then the shape and size of configuration is known. The two unit vectors $\mathbf{n}(t)$ together with $\frac{q_1(t)}{|q_1(t)|}$ are enough to determine $q(t)$. Since we are interest in the angle of rotation, it's convenient to use angle coordinates. Let $(\phi_n, \eta_n)$ be coordinate for $\mathbf{n}$ on the unit sphere, where $\phi_n\in [0,\pi]$ is the angle between $\mathbf{n}$ and $\mathbf{e}$, $\eta_n$ corresponds to the angle coordinate on the plane $X$. Note that $\frac{q_1(t)}{|q_1(t)|}$ is on a circle $S^1$ orthogonal to $\mathbf{n}$, then $\frac{q_1(t)}{|q_1(t)|}$ can be represented by a angle $\theta$ once $\mathbf{n}$ is given. Let $\eta$ be the angle coordinate in the plane $X$. For simplicity, we can choose coordinate system for $\theta$ so that: $\eta=\theta$ when $\mathbf{n}=\mathbf{e}$, the angle coordinate of $P(\theta,\mathbf{n})$ on $X$ equals to $\eta_n+\theta$ when $\phi_n\in (0,\pi)$. We have an illustration for the coordinate system in Figure \ref{3}.
\begin{figure}
	\centering
	\includegraphics[height=9.4cm,width=15.6cm]{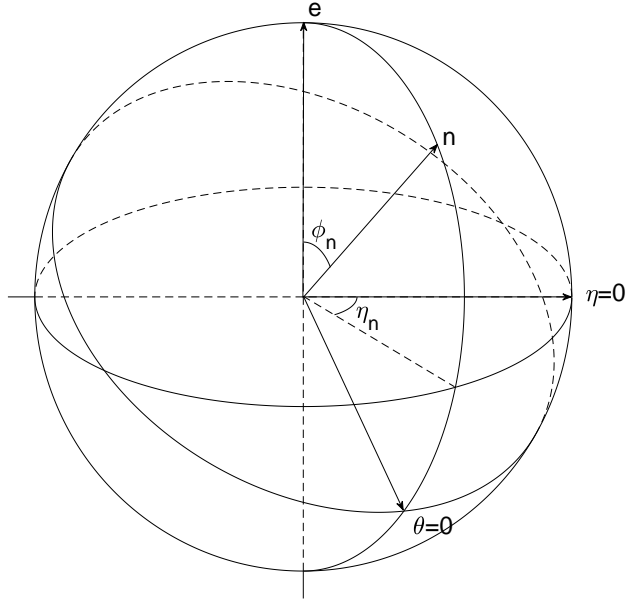}
	\caption{An illustration for $\phi_n, \eta_n$ and vectors defined by $\eta=0, \theta=0$}
	\label{3}
\end{figure} 

Now we have a path $(q_1,\mathbf{n})$ and $P(q_1,\mathbf{n})$ defines $x_1$. The aim is to compute $\eta_1(T)-\eta_1(0)$. Let $\theta_1$ be $\theta$ coordinate of $q_1$. By using coordinates defined above, it holds 
\begin{equation}
	\dot{\eta}_1(t)=\dot{\eta}_n(t)+\dot{\theta}_1(t).
\end{equation}

Denote
\begin{equation}\label{psi}
\Psi: q(t)\rightarrow (q_1(t),\mathbf{n}) \rightarrow (\phi_n(t),\eta_n(t),\theta_1(t)).
\end{equation}
As in the introduction section, the velocity can be decomposed into $V_R$ and $V_I$ part as in \eqref{eqdecompose}. We shall investigate how the two parts of velocities effects $\dot{\eta}_n(t)$ and $\dot\theta_1(t)$.

Let $v_I(t)$ the $V_I$ part velocity of $\dot{q}(t)$. Proposition \ref{progui} says that $v_I$ is tangent to the configuration plane, thus keeps the normal vector $\mathbf{n}$ and $\eta_n(t)$ fixed. But it effects $\theta_1(t)$ as in the planar case, which has been carefully studied in last section. Thus 
\begin{equation}\label{eqdotvi}
d\Psi(v_I)=(0,0,2\dot{R}(t)).
\end{equation}

Let $v_R(t)$ the $V_R$ part velocity of $\dot{q}(t)$. $v_R(t)$ can be seen as velocity given by a rigid rotation of $q(t)$ with angular momentum $J(t)$. 

When $\phi_n\in (0,\pi)$, $\mathbf{e},\mathbf{n},\mathbf{e}\wedge\mathbf{n}$ forms a basis of $\mathbb{R}^3$. By \eqref{psi} and the definition of $\phi_n,\eta_n,\theta_1$, one finds that $\dot\phi_n,\dot\eta_n$ and   $\dot\theta_1$ are exactly the angular velocities about $\mathbf{e}\wedge\mathbf{n},\mathbf{e}$ and $\mathbf{n}$. Whenever $q$ is a triangular configuration, $\sigma^{-1}$ is an isomorphism.  $\sigma^{-1}$  maps the angular momentum to a unique vector in $\mathbb{R}^3$ which give an axis and an angular velocity. It holds
\begin{equation}\label{eqdotvr}
d\Psi(v_R(t))=\Big(\sigma_{\wedge}^{-1}(J),\sigma_{\mathbf{e}}^{-1}(J),\sigma_{\mathbf{n}}^{-1}(J)\Big).
\end{equation}
where $\Big(\sigma_{\wedge}^{-1}(J),\sigma_{\mathbf{e}}^{-1}(J),\sigma_{\mathbf{n}}^{-1}(J)\Big)$ is the coordinate of $\sigma^{-1}(J)$ given by the direct decomposition about the basis $\mathbf{e},\mathbf{n},\mathbf{e}\wedge\mathbf{n}$.

By \eqref{eqdotvi} and \eqref{eqdotvr},
\begin{equation}
\begin{split}
(\dot\phi_n(t),\dot\eta_n(t),\dot\theta_1(t))&=d\Psi(v(t)),\\
&=d\Psi(v_I(t))+d\Psi(v_R(t)),\\
&=\Big(\sigma_{\wedge}^{-1}(J(t)),\sigma_{\mathbf{e}}^{-1}(J(t)),\sigma_{\mathbf{n}}^{-1}(J(t))+2\dot{R}(t)\Big).
\end{split}
\end{equation}
Hence
\begin{equation}\label{eqdotxix}
\dot\eta_1(t)=\dot{\eta}_n(t)+\dot{\theta}_1(t)=\sigma_{\mathbf{e}}^{-1}(J(t))+\sigma_{\mathbf{n}}^{-1}(J(t))+2\dot{R}(t).
\end{equation}

When $\mathbf{n}=\pm\mathbf{e}$, then $\eta_n$ is not well-defined. Angular velocity about $\mathbf{e}$ keeps $\mathbf{n}$ fixed. Only angular velocity about $\mathbf{n}$ contributes to $\eta_1$. In this case
\begin{equation}\label{eqneqe}
\dot\eta_1(t)=\dot{\theta}_1(t)=\mathbf{n}\cdot\sigma^{-1}(J(t))+2\dot{R}(t).
\end{equation}
$\mathbf{n}=\mathbf{e}$ is actually a limit case of \eqref{eqdotxix} when $\mathbf{n}\rightarrow\mathbf{e}$, while this does not hold for $\mathbf{n}=\mathbf{-e}$. 

We define 
\begin{equation}
F(J(t))=\left\{
\begin{array}{lrl}
\sigma_{\mathbf{e}}^{-1}(J(t))+\sigma_{\mathbf{n}}^{-1}(J(t)), &\text{if}& \phi_n\in (0,\pi),\\ 
\mathbf{n}\cdot\sigma^{-1}(J(t)), &\text{if}& \mathbf{n}=\pm\mathbf{e}.
\end{array} \right.
\end{equation}
Then $F(J(t))$ is the angular velocity of $x_1(t)$ given by $v_R(t)$. Hence
\begin{equation}\label{eqdoteta1}
\dot{\eta}_1(t)=F(J(t))+2\dot{R}(t).
\end{equation}

If \eqref{eqdoteta1} holds for any $t\in [0,T]$, by integrating \eqref{eqdoteta1} we have proved
\begin{equation}\label{eqdeltaeta}
\Delta\eta_1=\int_{0}^T\dot{\eta}_1(t) dt =\int_{0}^TF(J(t))dt+2\int_{D_\gamma}\Omega_0.
\end{equation}
where $D_\gamma$ is the same as in \eqref{reconstruction}. 

\eqref{eqdeltaeta} is the formula we are looking for. We have proved \eqref{eqdoteta1} whenever $q(t)$ is a triangular configuration for all $t\in [0,T]$. For collinear configurations, \eqref{eqdoteta1} does not always hold. To see this, let $q(t)$ be a collinear configuration and $\mathbf{e}_3$ be axis of configuration. Assume $\mathbf{e}\ne \mathbf{e}_3$, let $\mathbf{n}$ be a normal vector in the plane spanned by $\mathbf{e}$ and $\mathbf{e}_3$. We also assume that $\mathbf{e}\ne \mathbf{n}$. It's not hard to see that rotations of $q(t)$ about $\mathbf{e}$ and $\mathbf{n}$ with different angle velocities may have the same $\dot{q}(t)$ and $J(t)$. But they lead to different $\dot{\eta}_1(t)$ which is absurd. We shall see that \eqref{eqdoteta1} fails to hold exactly at those bad collinear configurations show in \eqref{eqbad}

\textbf{Claim:} If $q(t)$ is a collinear configuration and $\mathbf{e}$ is orthogonal to the axis of configuration, then \eqref{eqdoteta1} holds. 

Denote $\mathbf{e}_3$ to be the axis of $q(t)$. Clearly, $\sigma^{-1}(J(t))$ and $\sigma^{-1}(J(t))+s\mathbf{e}_3$ in $Lie(G)$ give the same $v_R(t)$ for any $s\in \mathbb{R}$. If $\mathbf{e}\perp \mathbf{e}_3$, then $\mathbf{e}_3$ is collinear with $\mathbf{e}\wedge\mathbf{n}$. Note that angular velocity about $\mathbf{e}\wedge\mathbf{n}$ does not contribute to $\dot\eta_1(t)$. Thus the angular velocity of $x_1(t)$ given by $v_R(t)$ is uniquely defined by $\sigma^{-1}(J(t))$. The claim is proved.

If $\mathbf{e}$ is not orthogonal to $\mathbf{e}_3$, then angular velocity about $\mathbf{e}_3$ will have nonzero direct decomposition about $\mathbf{e}$ and $\mathbf{n}$. Thus the angular velocity of $x_1(t)$ given by $v_R$ is not uniquely defined. When $J(t)=0$, there is no $v_R$ thus only $v_I$ contribute to $\dot{\eta}_1$. For these reasons, we define the following set
\begin{equation}\label{eqbad1}
\Delta=\Big\{t\in [0,T]\Big|q(t)\  \text{is collinear} , J(t)\ne 0, \mathbf{e}\  \text{is not orthognal to the axis of configuration} \Big\}
\end{equation}
$\Delta$ is the set when \eqref{eqdoteta1} fails to hold. If $\Delta$ has zero measure, then \eqref{eqdeltaeta} still holds. Thus we complete the proof of our main result if $x(t)$ is smooth. 

Now it suffices to show that $x(t)$ is smooth under assumptions $(i)$ and $(ii)$. Clearly, $P(q(t),\mathbf{n}(t))$ is well-defined and continuous whenever $\mathbf{n}(t)\ne-\mathbf{e}$. If there exist some $t_0$ such that $\mathbf{n}(t_0)= -\mathbf{e}$, then $P(q_1,\mathbf{n})$ is not well-defined at $t_0$. We should consider $\lim_{t\rightarrow  t_0^{\pm}}P(q_1,\mathbf{n})$. In angle coordinates $(\phi, \eta, \theta)$, $P$ is defined in $[0,\pi)\times S^1\times S^1$. It can be extend to $(\phi,\eta,\theta)\in \pi\times S^1\times S^1$ continuously. Note that we regards $(\phi, \eta, \theta)$ as an unit vector representing $q_1$ and $\mathbf{n}$. By our choice of coordinate system, it holds
\[ (\pi, \eta_n, \theta_1)= (\pi, \eta_n+\alpha, \theta_1+\alpha),\ \ \forall \alpha\in \mathbb{R}.\]
Denote
\begin{equation}\label{Pextension}
\begin{split}
&\eta_n^{\pm}=\lim_{t\rightarrow  t_0^{\pm}}\eta_n,\ \ \ \ \theta_1^{\pm}=\lim_{t\rightarrow  t_0^{\pm}}\theta_1 \\
&P(\pi,\eta_n^{\pm},\theta_1^{\pm})=\lim_{t\rightarrow  t_0^{\pm}}P(q_1,\mathbf{n})=\lim_{t\rightarrow  t_0^{\pm}}P(\phi_n,\eta_n,\theta_1).
\end{split}
\end{equation}

If $\mathbf{n}(t_0)=\mathbf{-e}, \dot{\mathbf{n}}(t_0)\ne 0$, clearly $\eta_n^+=\eta_n^-+\pi$, $\theta_1^+=\theta_1^-+\pi$ and
\[(\pi,\eta_n^{+},\theta_1^{+})=(\pi,\eta_n^{-},\theta_1^{-})=(q_1(t_0), \mathbf{n}(t_0)).\]
We have
\[\lim_{t\rightarrow t_0^+}\eta_1(t)=\eta_n^++\theta_1^+=\eta_1^-+\theta_1^-+2\pi=\lim_{t\rightarrow t_0^-}\eta_1(t)+2\pi.\]
Thus 
\begin{equation}
P(\pi,\eta_n^{+},\theta_1^{+})=P(\pi,\eta_n^{-},\theta_1^{-}).
\end{equation}
$t_0$ is actually a removable singularity of $P(q_1,\mathbf{n})$ and $x_1(t)$ can be made smooth at $t_0$. Thus \eqref{eqdeltaeta} still holds in this case. 

The proof of Theorem \ref{thmspatial} is complete.

\qed

\begin{remark}
	The assumption that $\Delta$ has a zero measure can not be dropped. We can see this by considering following motions. Let $q(0)$ be a collinear configuration and $\mathbf{e}_3$ be axis of configuration. Assume $\mathbf{e}\ne \mathbf{e}_3$, let $\mathbf{n}$ be a normal vector in the plane spanned by $\mathbf{e}$ and $\mathbf{e}_3$. Now consider two motions defined by pure rotations of $q(0)$ about $\mathbf{e}$ and $\mathbf{n}$, then the corresponding motions in $X$ is also pure rotations with different angular velocities. But $\eqref{reconstructionspatial}$ gives the same $\dot{\eta}_1(t)$ for the two motions. 
\end{remark}

\section{On Montgomery's question}
In \cite{Mont1996}, Montgomery proved \eqref{eqmont} when the angular momentum is preserved and the initial and final configuration are similar triangles. He asked the following question: can reconstruction formula and calculations be reformulated solely in terms of $\Sigma/SO(2)$? In this section, we apply \eqref{reconstructionspatial} to Montgomery's case and give a positive answer to this question.

In his situation, the angular momentum is preserved and in the direction of $\mathbf{e}$. Clearly, angular momentum at a collinear configuration must be orthogonal to the axis of configuration. Thus $\Delta$ is empty. Theorem \ref{thmspatial} can be applied directly to this case. 

Note that the $SO(2)$ action is defined by rotations about $\mathbf{e}$. It's clear that the second term in the right of  \eqref{reconstructionspatial} depends only on oriented similar classes of configuration. It's calculation is actually reduced to the shape sphere, which is a subset of $\Sigma/SO(2)$. To show that reconstruction formula and calculations be reformulated solely in terms of $\Sigma/SO(2)$, it suffices to prove the following proposition.

\begin{proposition}\label{promont}
	If $J(t)=j(t)\cdot\mathbf{e}$, where $j(t)$ is a real valued function of $t$, then $\sigma_{\mathbf{e}}^{-1}(J(t))+\sigma_{\mathbf{n}}^{-1}(J(t))$ only depends on the class of $q$ in $\Sigma/SO(2)$.
\end{proposition}
\begin{proof}
	Note that $\sigma$ depends on $q$. Given any $R\in SO(2)$, denote $q'=R(q)$, $J'(t)=R(J(t))$, $\mathbf{n'}=R(\mathbf{n})$. Let $\sigma$ and $\sigma'$ be the corresponding maps from $Lie(G)$ to angular momentum. Clearly, there is 
	\[R\big(\sigma^{-1}(J(t))\big)=(\sigma')^{-1}(J'(t)).\]
	Since $J(t)=j(t)\cdot\mathbf{e}$, it holds $J'(t)=J(t)$. Thus
	\begin{equation}\label{eqsig}
    (\sigma')^{-1}(J(t))=R\big(\sigma^{-1}(J(t))\big).
    \end{equation}	
    We are interested in the direct decomposition of $(\sigma')^{-1}(J(t))$ according to a basis depending on $q'$. By \eqref{eqsig}, $(\sigma')^{-1}(J(t))$ and $\sigma^{-1}(J(t))$ differ by the action $R$, so are the corresponding bases. Thus
    \begin{equation}
    \Big(\sigma_{\wedge}^{-1}(J(t)),\sigma_{\mathbf{e}}^{-1}(J(t)),\sigma_{\mathbf{n}}^{-1}(J(t))\Big)=\Big((\sigma')_{\wedge'}^{-1}(J(t)),(\sigma')_{\mathbf{e}}^{-1}(J(t)),(\sigma')_{\mathbf{n'}}^{-1}(J(t))\Big).
    \end{equation}
    This completes the proof.

\end{proof}

\end{document}